\documentclass[a4paper,12pt]{amsart}
\usepackage[utf8]{inputenc}
\usepackage[T1]{fontenc}
\usepackage{fullpage}
\usepackage[english]{babel}
\usepackage{mathrsfs}
\usepackage{amsmath, amsthm, amssymb}
\usepackage[retainorgcmds]{IEEEtrantools} 
\title{Characterizations of Hardy-Orlicz spaces of Quasiconformal Mappings}
\author{Sita Benedict}
\newtheorem{theorem}{Theorem}[section]
\newtheorem{lemma}[theorem]{Lemma}

\newtheorem*{thma}{Theorem A}

\newtheorem*{theorem*}{Theorem}
\newtheorem*{proposition*}{Proposition}
\newtheorem*{corollary*}{Corollary}

\newtheorem*{remark*}{Remark}
\numberwithin{equation}{section}
\newcommand{\Sn}{\mathbb{S}^{n-1}} 
 
\newcommand{\Bn}{\mathbb{B}^{n}} 
\newcommand{\Rn}{\mathbb{R}^{n}} 
\newcommand{\omn}{\omega_{n-1}}
\newcommand{\func}{f:\Bn \rightarrow \Rn}
\begin{document}
\begin{abstract}
An $H^p$-theory of quasiconformal mappings on $\Bn$ has already been established. By replacing $t^p$ with a general increasing growth function $\psi(t)$ we define the Hardy-Orlicz spaces of quasiconformal mappings and prove various characterizations of these spaces.
\end{abstract}

\footnotetext{
\endgraf The author was partially supported by the Academy of Finland grants
131477 and 263850.}

\maketitle
\section{Introduction}
Hardy-Orlicz spaces are a natural generalization of the Hardy spaces. Holomorphic functions on the unit disk in $\mathbb{C}$ belonging to Hardy-Orlicz spaces have been studied in \cite{stoll1}, \cite{khalil}, \cite{deeb1}, and \cite{deeb2}.  For the higher dimensional case of holomorphic maps on the unit ball in $\mathbb{C}^n$ see for example \cite{hardyorlicz}, \cite{pcarleson} and \cite{charpentier}. In this paper we are interested in the generalization of Hardy spaces of quasiconformal mappings on the unit ball in $\Rn$. 

A quasiconformal mapping $f:\Bn \rightarrow \Rn$ belongs to the Hardy space $H^p$ for a fixed $0 < p < \infty$ if the values $\int_{\Sn} |f(r\omega)|^p d\sigma$ are uniformly bounded for all $r\in [0,1)$. Here $\sigma$ is the surface measure on $\mathbb{S}^{n-1}$. For results on these spaces see especially \cite{hpqc} and also \cite{nolder} and \cite{noldercarleson}. We highlight in particular the following characterization theorem, proved as several results in \cite{hpqc}. 
\begin{thma}
Let $f$ be a quasiconformal mapping of $\Bn$, $f_i$ one of its component functions and fix $0 < p < \infty$. Then the following are equivalent:
\begin{enumerate}
\item $f \in H^p$
\item $f(\omega) \in L^p(\Sn)$
\item $f^*(\omega) \in L^p(\Sn)$
\item $f_i^*(\omega) \in L^p(\Sn)$
\item $\int_0^1 (1-r)^{n-2}M(r,f)^pdr < \infty$
\item $\int_{\Bn} a_f^p(x)(1-|x|)^{p-1}dx < \infty$.
\end{enumerate}
\end{thma}

The function $a_f(x)$ is an averaged version of the differential $Df(x)$, see Section 2 for its definition. The equivalence of (1) and (6) in Theorem A is the quasiconformal version of the following area characterization for $f$ conformal:
\begin{eqnarray*}
f\in H^p \;\;\;\textnormal{if and only if} \;\;\int_{\mathbb{B}^2}|f'(x)|^p(1-|x|)^{p-1} dx < \infty,
\end{eqnarray*}
see \cite{dirichlet}. The non-tangential maximal function $f^*$ and maximum modulus $M(r,f)$ are defined in Section 3. 

Our main results generalize Theorem A to Hardy-Orlicz spaces of quasiconformal mappings. Let $\psi$ be a \textit{growth function}; that is, a differentiable and strictly increasing function mapping $[0, \infty]$ to itself such that $\psi(0) = 0$. Then a quasiconformal mapping $f: \Bn \rightarrow \Rn$ belongs to the Hardy-Orlicz space $H^\psi$ if there exists $\delta > 0$ such that
\begin{align*}\label{hpsidef}
\sup_{0<r<1}\int_{\mathbb{S}^{n-1}} \psi(\delta|f(r\omega)|)d\sigma(\omega) < \infty.
\end{align*}

Our first result shows that much of Theorem A extends to the Hardy-Orlicz setting without any additional restrictions on the growth of $\psi$. 
\begin{theorem}\label{equivthm}
Given a growth function $\psi$ and a quasiconformal mapping $f:\Bn \rightarrow \Rn$ the following are equivalent:
\begin{enumerate}
\item $f(x) \in H^\psi$
\item $\psi(\delta_1|f(\omega)|) \in L^1(\Sn)$\;\; \text{for some $\delta_1 >0$}
\item $\psi(\delta_2f^*(\omega)) \in L^1(\Sn)$\;\; \text{for some $\delta_2 >0$}
\item $\int_0^1 (1-r)^{n-2}\psi(\delta_3 M(r,f)) dr < \infty$ \;\; \text{for some $\delta_3 >0$}.
\end{enumerate}
\end{theorem}

The characterizations involving $a_f(x)$ and the component function $f_i$ are extended to the Hardy-Orlicz spaces that have a doubling condition on the growth function $\psi$ and its inverse.
\begin{theorem}\label{equivdoubling}
Let $f$ be a quasiconformal mapping of $\Bn$, $f_i$ one of its component functions and $\psi$ a growth function such that both $\psi$ and $\psi^{-1}$ are doubling. Then the following are equivalent:
\begin{enumerate}
\item $f(x) \in H^\psi$
\item $\int_{\Bn} \psi( a_f(x) (1-|x|)) \frac{dx}{1-|x|} < \infty$
\item $\psi(f_i^*(\omega)) \in L^1(\Sn)$
\end{enumerate}
\end{theorem} 
The equivalences in Theorem \ref{equivdoubling} can fail if either $\psi$ or $\psi^{-1}$ is not doubling. For example, if $f(z) = z$ then $f$ belongs to every Hardy-Orlicz space $H^\psi$.  However, we can construct a growth function $\psi$ such that the integral $\int_{1/2}^1 \frac{\psi(1-r)}{1-r}dr$ diverges, thus failing the implication (1) $\Rightarrow$ (2) for this $f$. The inverse of the growth function from our example is not doubling. The implications (2) $\Rightarrow$ (1) and (3) $\Rightarrow$ (1) both fail for $f(z) = \log(z +1)$ and an appropriate growth function $\psi$ that is not doubling. See section 4 for details.

This paper is organized as follows. Section 2 includes notation, definitions and background lemmas necessary for the proofs of our main results. Section 3 focuses on the characterizations of Hardy-Orlicz spaces that hold for all growth functions, and Section 4 gives the results that hold for $\psi$ satisfying additional growth conditions. 
\section{Preliminaries}
We denote by $B^n(x,r)$ the open ball in $\Rn$ of radius $r$ centered at $x$ and write its boundary as $S^{n-1}(x,r)$. We abbreviate $B^n(0, 1) = \Bn$, $S^{n-1}(0,1) = \Sn$ and let $\omega_{n-1}$ denote the surface measure of $\Sn$. For each $x\in \Bn$ we set $B_x = B^n(x, (1-|x|)/2)$ and then define the cap $S_x = \{\frac{y}{|y|}: y \in B_x, y\neq 0 \}$. Given $\omega \in \Sn$ let
\begin{align*}
\Gamma(\omega) = \bigcup\{B_{t\omega} : 0 \leq t < 1\}
\end{align*}
be a Stolz cone at $\omega$. Clearly $x\in \Gamma(\omega)$ if and only if $\omega \in S_x$. 

When a constant is written as $C = C(a,b,...)$ it means that the value of $C$ depends only on the values of $a, b, ...$. The values of constants may change from line to line in a sequence of inequalities without explicit mention or special notation. We use the symbol $A \approx B$ to denote that there exists a constant $C$ such that
\begin{align*}
\frac{A}{C} \leq B \leq CA.
\end{align*}

A homeomorphism of a domain $\Omega$ in $\Rn$ into $\Rn$ is $K$-quasiconformal if $f$ belongs to the Sobolev class $W_{loc}^{1,n}(\Omega, \Rn)$ and $|Df(x)|^n \leq KJ_f(x)$ for almost every $x \in \Omega$. In this paper all quasiconformal mappings will have as domain $\Bn$. 

The quasiconformal analogue of Beurling's theorem, \cite[Theorem~4.4]{conformalmetrics}, says that given a quasiconformal mapping $f: \Bn \rightarrow \Rn$ the radial limit
\begin{eqnarray*}
f(\omega):= \lim_{r\rightarrow 1} f(r\omega)
\end{eqnarray*}
exists for almost every $\omega \in \mathbb{S}^{n-1}$. 

An important tool for us will be the modulus of curve families. Given a family of curves $\Gamma$ in $\Rn$ the modulus Mod($\Gamma) \in [0, \infty]$ is
\begin{eqnarray*}
\text{Mod}(\Gamma) = \inf \int_{\Rn} \rho^n\; dm,
\end{eqnarray*}
where the infimum is taken over all admissible Borel functions $\rho:\Rn\rightarrow [0,\infty]$. A non-negative Borel function $\rho$ on $\Rn$ is considered admissible if ${\int_\gamma \rho ds \geq 1}$ for each locally rectifiable $\gamma \in \Gamma$. We collect some basic results about modulus of curve families here; their proofs can be found in \cite{vaisala}. 

First, the modulus of a curve family is quasi-invariant under quasiconformal mappings. More precisely, if $f: \Omega \rightarrow \Rn$ is $K$-quasiconformal then ${\text{Mod}(\Gamma)/K \leq \text{Mod}(f\Gamma) \leq K\text{Mod}(\Gamma)}$ for every family of curves $\Gamma \subset \Omega$. Here $f\Gamma = \{f\circ \gamma : \gamma \in \Gamma \}$.

The exact value of the modulus can be calculated for certain families of curves. Let $\Gamma$ be the collection of radial segments joining $S(0, r), 0 < r < 1$, to a Borel set $E \subset \Sn$. Then
\begin{align*}
\text{Mod}(\Gamma) = \sigma(E)(\log(1/r))^{1-n}.
\end{align*}
Moreover, if $\Gamma$ is a family of curves with each $\gamma \in \Gamma$ joining $S^{n-1}(x, r)$ to $S^{n-1}(x, R)$, $0 < r < R$, then we have the upper bound
\begin{align*}
\text{Mod}(\Gamma) \leq \frac{\omn}{(\log(R/r))^{n-1}}.
\end{align*}

We will make repeated use of the following lemma (cf. \cite[Theorem 18.1] {vaisala}), which is a direct result of modulus estimates. 
\begin{lemma}\label{firstmoduluscons}
Let $f: \Bn \rightarrow \Omega$ be $K$-quasiconformal. There is a constant $C$ depending only on $n$ and $K$ so that for each $x \in \Bn$
\begin{eqnarray*}
\textnormal{diam}(f(B_x))/C \leq d(f(x), \partial \Omega) \leq C \textnormal{diam}(f(B_x)).
\end{eqnarray*}
Moreover, $f(B_x)$ contains a ball of radius $d(f(x), \partial \Omega)/C$, centered at $f(x)$. 
\end{lemma}

The following lemmas are proved similarly using modulus estimates. 
\begin{lemma}\label{Mineq}
Let $f$ be a $K$-quasiconformal mapping of $\Bn$ onto $\Omega \subseteq \Rn$. For each $x\in \Bn$ and $M > 1$
\begin{eqnarray*}
\sigma(\{\omega \in S_x : |f(\omega) - f(x)| > M d(f(x), \partial\Omega) \}) \leq C \sigma(S_x)(\log M) ^{1-n},
\end{eqnarray*}
where the constant $C$ depends only on $n,K$. 
\end{lemma}
\begin{proof}
Abbreviate $d = d(f(x), \partial\Omega)$ and $E = \{\omega \in S_x : |f(\omega) - f(x)| > M d\}$ and first assume that $|x| < 1/4$. Let $\Gamma_E$ be the set of radial segments with one endpoint in $E$ and the other in $B_x \cap S(0,1/4)$. Then Mod$(\Gamma_E) = \sigma(E)(\log 4)^{1-n}$ and by Lemma \ref{firstmoduluscons} there exists a constant $C = C(n, K)$ such that each curve in $f(\Gamma_E)$ joins $S(f(x), Cd)$ with $S(f(x), Md)$. If $2 \leq C$ and $C^2 < M$ then
\begin{eqnarray*}
\sigma(E)(\log 4)^{1-n} = \textnormal{Mod}(\Gamma_E) \leq K\textnormal{Mod}(f(\Gamma_E)) \leq C\omn (\log M)^{1-n}
\end{eqnarray*}
and if $1 < M \leq C^2$ then we have anyway
\begin{eqnarray*}
\sigma(E) \leq \omn(\log C^2)^{n-1}(\log M)^{1-n}.
\end{eqnarray*}
For the other case when $1/4 \leq |x|$ set $\Gamma_E$ to be the collection of radial segments with one endpoint in $E$ and the other endpoint in $B_x \cap S(0, |x|)$. Then Mod$(\Gamma_E) = \sigma(E)(\log \frac{1}{|x|})^{1-n}$ and as above each curve in $f(\Gamma_E)$ joins $S(f(x), Cd)$ with $S(f(x), Md)$. Calculating like before, the case $1 < M \leq C^2$ is trivial and assuming $C^2 < M$ we get that
\begin{eqnarray*}
\sigma(E)(\log 1/|x|)^{1-n} \leq C \omn (\log M)^{1-n}.
\end{eqnarray*}
Noting that $(\log 1/|x|)^{n-1} \approx \sigma(S_x)$ we are done. 
\end{proof}

\begin{lemma}\label{lemma42restatement}
Let $\func$ be $K$-quasiconformal with $f(x) 
\neq 0$ for all $x\in \Bn$, $\phi$ a growth function and $\delta >0$.  There is an absolute constant $C$ such that for each $x \in \Bn$ and $M > 1,$
\begin{eqnarray*}
\sigma(\{\omega \in S_x : \phi(\delta|f(\omega)|) < \phi(\delta|f(x)|/M)\}) \leq C\sigma(S_x)(\log M)^{-1}.
\end{eqnarray*}
\end{lemma}
\begin{proof}
Let $x \in \Bn$ and $\omega \in S_x$. It suffices to prove the inequality in the instance that $\delta = 1$ and $\phi$ is the identity map on $[0, \infty]$, since growth functions are strictly increasing. 

Set $E = \{\omega \in S_x :  |f(\omega)| < |f(x)|/M\}$, and choose the curve families $\Gamma_E$ like in the proof of Lemma~\ref{Mineq}, taking separately the cases $|x| \leq 1/4$ and $1/4 < |x| < 1$. Each curve belonging to $f(\Gamma_E)$ will have one endpoint in $B(f(x), C|f(x)|)$ and the other in {$\mathbb{R}^n \setminus B(f(x), |f(x)|/M)$}, for some absolute constant $C$. The desired upper bound follows using the same modulus of curve family techniques as in the proof of Lemma~\ref{Mineq}. 
\end{proof}

The average derivative of a $K$-quasiconformal mapping, introduced by Astala and Gehring \cite{analogues}, is defined as
\begin{align*}
a_f(x) = \exp(\int_{B_x}\log J_f(y) \frac{dm}{n|B_x|}).
\end{align*}
Here $|B_x|$ is the $n$-measure of $B_x$. The mean value property implies that $a_f = |Df|$ if $f$ is conformal. The following lemma is proved in \cite{analogues}, and it is an example of how $a_f(x)$ can take the place of $|f'(x)|$ in quasiconformal analogues of statements originally proved for conformal mappings on the unit disk.
\begin{lemma}\label{afequiv}
Let $f: \Bn \rightarrow \Omega$ be $K$-quasiconformal. There is a constant $C$ depending only on $n,K$ so that for each $x \in \Bn$
\begin{align*}
d(f(x), \partial\Omega)/C \leq a_f(x)(1-|x|)\leq C d(f(x), \partial\Omega)
\end{align*}
and
\begin{align*}
\frac{1}{C}\left(\frac{1}{|B_x|}\int_{B_x} |Df(y)|^ndm \right)^{1/n} \leq a_f(x) \leq C\left(\frac{1}{|B_x|}\int_{B_x} |Df(y)|^ndm \right)^{1/n}.
\end{align*}
\end{lemma}
The following is a consequence of Lemmas \ref{firstmoduluscons} and \ref{afequiv}. See \cite[Lemma 2.5]{hpqc} for details.
\begin{lemma}\label{afintegralequiv}
Let $f: \Bn \rightarrow \Omega$ be $K$-quasiconformal. Suppose that $u > 0$ satisfies
\begin{align*}
u(x) \approx u(y)
\end{align*}
for each $x\in \Bn$ and all $y\in B_x$. Let $0 < q \leq n$ and $p\geq q$. Then
\begin{align*}
\int_{\Bn} a_f^p(x) u(x) dx \approx \int_{\Bn} a_f(x)^{p-q}|Df(x)|^qu(x)dx.
\end{align*}
\end{lemma}
\section{Characterizations of $H^\psi$}
With each quasiconformal mapping $f:\Bn \rightarrow \Rn$ we associate its maximum modulus function 
\begin{align*}
M(r,f) = \{\sup |f(x)| : |x| = r \}, \;\;\;r \in (0, 1).
\end{align*}
\begin{theorem}\label{firstequiv} 
Let $f$ be a $K$-quasiconformal mapping of $\Bn$ and $\psi$ a growth function. Then the following are equivalent:
\begin{eqnarray}\label{Lpsicond}
\psi(\delta_1 |f(\omega)|)\in L^1(\Sn)\;\; \text{for some $\delta_1 >0$}.
\end{eqnarray}
\begin{eqnarray}\label{maxfunchar}
\int_0^1 (1-r)^{n-2}\psi(\delta_2 M(r,f)) dr < \infty\;\; \text{for some $\delta_2 >0$}.
\end{eqnarray}
\end{theorem}
\begin{proof}
First suppose (\ref{maxfunchar}) holds for some $\delta > 0$. We can assume that $f(0) = 0$.  We will show in this case that there exists a constant $C = C(n,K)$ such that
\begin{eqnarray}\label{distequiv}
\int_{\Sn} \psi(\frac{\delta}2|f(\omega)|) d\sigma \leq C\int_0^1(1-r)^{n-2}\psi(\delta M(r,f)) dr.
\end{eqnarray}

To prove (\ref{distequiv}) we rewrite the integral on the left as
$$
\int_{\mathbb{S}^{n-1}} \psi(\frac{\delta}2|f(\omega)|)d\sigma(\omega) = \int_0^\infty \psi'(\lambda) \sigma(\{\omega \in \Sn: \frac{\delta}2|f(\omega)| > \lambda \})d\lambda.
$$
Let $E = \{\omega \in \Sn:  \frac{\delta}2|f(\omega)| > \lambda \}$ for a fixed $\lambda >0$. We obtain an upper bound on $\sigma(E)$ using modulus of curve families as follows.

There exists a unique $r = r(\lambda)$ such that
$$
\delta M(r,f) = \lambda.
$$
Let $\Gamma_E$ be the family of radial segments connecting $B(0, r)$ to $E$. Then
$$
M(\Gamma_{E}) = \frac{\sigma(E)}{(\log(1/r))^{n-1}} \geq \frac{\sigma(E)}{2^{n-1}(1-r)^{n-1}}
$$
as long as $1/2 < r < 1.$ 

Each curve in $ f(\Gamma_{E})$ connects $B(0,\lambda/\delta)$ to $\mathbb{R}^n\setminus B(0,2\lambda/\delta)$, and so 
$$
M(f\Gamma_{E}) \leq \frac{\omn}{(\log2)^{n-1}}.
$$
Since $M(\Gamma_{E}) \leq K M(f\Gamma_{E})$ we have
$$
\sigma(E) \leq C(n,K)(1-r)^{n-1}
$$
whenever $\delta M(r,f) = \lambda$ and $1/2 < r < 1$. This estimate and an application of Fubini's theorem now give
\begin{gather*}
 \int_0^\infty \psi'(\lambda) \sigma(\{\omega \in \Sn: \frac{\delta}2|f(\omega)| > \lambda \})d\lambda \leq \\ \leq
 \sigma(\Sn)\psi(\delta M(1/2,f)) + C(n,K)\int_0^\infty\psi'(\lambda)(1-r(\lambda))^{n-1}d\lambda = \\= 
  \sigma(\Sn)\psi(\delta M(1/2,f)) + C(n,K)\int_0^\infty\psi'(\lambda)\int_{r(\lambda)}^1(1-t)^{n-2}dtd\lambda = \\= 
    \sigma(\Sn)\psi(\delta M(1/2,f)) + C(n,K)\int_0^1(1-t)^{n-2}\int_0^{\delta M(t,f)} \psi'(\lambda)d\lambda dt \leq \\ \leq
    C(n,K)\int_0^1 (1-t)^{n-2}\psi(\delta M(t,f))dt,
 \end{gather*}
 which gives (\ref{distequiv}).

For the converse direction assume there exists $\delta > 0$ such that 
\begin{align*}
\int_{\mathbb{S}^{n-1}} \psi(\delta|f(\omega)|)d\sigma(\omega) < \infty,
\end{align*}
and choose points $x_k \in \Bn$ with ${|x_k| = r_k = 1 - 2^{-k}}$ and $|f(x_k)| = M(r_k, f), k=1,2,\mathellipsis$. 

Given any $\epsilon > 0$ we have
\begin{align*}
\int_0^1 (1-r)^{n-2}\psi(\epsilon M(r,f))dr &\leq& 2^n\sum_{k=1}^\infty (2^{-k})^{n-1}\psi(\epsilon M(r_k, f)) = 2^n\int_{\Bn} \psi(\epsilon |f(x)|)d\mu,
\end{align*} 
where $d\mu(x) = \sum_{k=1}^\infty (1-|x|)^{n-1}\delta_{x_k}$. The measure $\mu$ is clearly a Carleson measure, and assuming $f(x) \neq 0$ for all $x\in \Bn$ we can apply Lemma~\ref{carleson} below to obtain constants $C_1$ and $C_2$ not depending on $f$ or $\delta$ such that
\begin{eqnarray}\label{ineqref}
\int_{\Bn} \psi(\delta |f(x)|/C_1)d\mu \leq C_2\int_{\Sn} \psi(\delta |f(\omega)|) d\sigma.
\end{eqnarray}
By choosing $\epsilon = \delta/C_1$ the proof is finished in this case. The case when $0\in f(\Bn)$ is handled by applying the result to $g(x) = f(x) - y_0$ for some fixed $y_0 \in \Rn\setminus f(\Bn)$.\end{proof}

We now give some results involving maximal functions, which we need for proving Lemma \ref{carleson}. Given a quasiconformal mapping $f$ on $\Bn$ its nontangential maximal function is defined as
\begin{align*}
f^*(\omega) = \sup_{x\in \Gamma(\omega)} |f(x)|, \;\;\omega \in \Sn.
\end{align*}
Clearly $\psi(\delta f^*(\omega)) \in L^1(\Sn)$ implies $\psi(\delta |f(\omega)|) \in L^1(\Sn)$. The Hardy-Littlewood maximal function and one of the modulus estimates from Section 2 help us prove the reverse implication as stated in Theorem \ref{maximalstarequiv}.

\begin{theorem}\label{maximalstarequiv}
Let $\psi$ be a growth function and $f: \Bn \rightarrow \Rn$ a $K$-quasiconformal mapping such that $0 \notin f(\Bn)$. There exist constants $C_1 = C_1(n,K)$ and $C_2 = C_2(n)$ such that  
\begin{eqnarray*}
\int_{\Sn} \psi(\delta/C_1 f^*(\omega)) d\sigma \leq C_2 \int_{\Sn} \psi(\delta |f(\omega)|) d\sigma
\end{eqnarray*}
whenever $\delta > 0$.
\end{theorem}
\begin{proof}
We may assume there is $\delta > 0$ such that $\int_{\Sn} \psi(\delta |f(\omega)|) d\sigma < \infty.$ Let $\phi = \psi^{1/2}$. By Lemma \ref{lemma42restatement} there is a constant $C_1 = C_1(n,K)$ such that
\begin{align*}
\sigma(\{\omega\in S_x: \phi(\delta |f(\omega)|) \geq \phi(\delta|f(x)|/C_1)  \}) \geq \sigma(S_x)/2
\end{align*}
for every $x\in \mathbb{B}^n$. Thus
\begin{eqnarray*}
\int_{S_x} \phi(\delta|f(\omega)|) d\sigma &\geq& \phi(\delta|f(x)|/C_1) \sigma(\{\omega\in S_x: \phi(\delta |f(\omega)|)\geq \phi(\delta|f(x)|/C_1)  \})\\ &\geq& \phi(\delta|f(x)|/C_1) \frac{\sigma(S_x)}2
\end{eqnarray*}
for each $x\in \Bn$. Let $M$ denote the non-centered Hardy-Littlewood maximal function on $\Sn$; that is,
\begin{align*}
Mg(\omega) = \sup_{r>0} \frac{1}{|\Sn \cap B(\omega, r)|}\int\limits_{\Sn \cap B(\omega, r)} g d\sigma,
\end{align*}
where $g \in L^1(\Sn)$.
It follows from the previous inequality that
\begin{eqnarray*}
 \phi(\delta f^*(\omega)/C_1)  \leq 2M(\phi(\delta|f|)(\omega).
\end{eqnarray*}
Since $M$ is a bounded operator on $L^2(\Sn)$ \cite{stein} we have that
\begin{eqnarray*}
 \int_{\Sn} \psi(\delta f^*(\omega)/C_1) d\sigma = \int_{\Sn} \phi^2(\delta f^*(\omega)/C_1) d\sigma \leq 4\int_{\Sn} M^2(\phi(\delta |f|)(\omega)d\sigma \\\leq C_2\int_{\Sn} \phi^2(\delta |f(\omega)|) d\sigma = C_2\int_{\Sn} \psi(\delta |f(\omega)|) d\sigma,
\end{eqnarray*}
which completes the proof.
\end{proof}
We introduce Carleson measures in order to finally prove the lemma we used in the proof of Theorem \ref{firstequiv}. A measure $\mu$ on $\Bn$ is called a \textit{Carleson measure} if there exists a constant $C(\mu)>0$ such that
\begin{align*}
\mu(\Bn \cap B(\omega, r)) \leq C(\mu) r^{n-1}
\end{align*}
for all $\omega \in \Sn$ and all $r > 0$. The infimum of all such constants $C(\mu)$ is called the Carleson norm of $\mu$ and is denoted as $\alpha_\mu$. 

\begin{lemma}\label{carleson}
Let $\psi$ be a growth function, $f: \Bn \rightarrow \Rn$ a $K$-quasiconformal mapping such that $f(x) \neq 0$ for all $x \in \Bn$, and $\mu$ a Carleson measure on $\Bn$. Then there exist constants $C_1 = C_1(n,K)$ and $C_2 = C_2(\alpha_\mu, n)$ such that 
\begin{eqnarray*}
\int_{\Bn}\psi(\delta|f(x)|/C_1)d\mu \leq C_2\int_{\Sn} \psi(\delta |f(\omega)|) d\sigma
\end{eqnarray*}
whenever $\delta > 0$.
\end{lemma}
\begin{proof}
First let $\epsilon > 0$ be arbitrary and set $E(\lambda) = \{x \in \Bn : \epsilon|f(x)|> \lambda \}$ and $U(\lambda) = {\{\omega \in \Sn : \epsilon f^*(\omega) > \lambda \}}$. Then $U(\lambda)$ is an open set and we can use the generalized form of the Whitney decomposition \cite[Theorem III.1.3]{whitneyref} to write
\begin{eqnarray*}
U(\lambda) = \bigcup_{k=1}^\infty S_{x_k}
\end{eqnarray*}
where the points $x_k \in \Bn$ are chosen so that each $\omega \in U(\lambda)$ is contained in at most $N = N(n)$ caps $S_{x_k}$ and $(1 - |x_k|)/C \leq d(S_{x_k}, \partial U(\lambda)) \leq C(1 - |x_k|)$. Here $C$ is an absolute constant and the distance is measured in the spherical distance on $\Sn$. 

If $\epsilon |f(x)| > \lambda$, then $\epsilon f^*(\omega) > \lambda$ for all $\omega \in S_x$, so $E(\lambda)$ is contained  in the union of $B(x_k/|x_k|, C(1 - |x_k|)), k = 1, 2\dots.$ where $C$ is an absolute constant. Hence by the properties of the measure $\mu$ and the decomposition of $U(\lambda)$ we get
\begin{eqnarray*}
\mu(E(\lambda)) &\leq& \sum_{k=1}^\infty \mu(B(x_k/|x_k|, C(1 - |x_k|))\cap\Bn)\\ &\leq& C \sum_{k=1}^\infty \sigma(S_{x_k}) \leq C\sigma(U(\lambda)).
\end{eqnarray*}
Here $C$ depends on $n$ and the Carleson norm of $\mu$. 

Therefore,
\begin{eqnarray*}
\int_{\Bn} \psi (\epsilon|f(x)|) d\mu&=& \int_0^\infty \psi'(\lambda) \mu(E(\lambda))d\lambda \\ &\leq& C \int_0^\infty  \psi'(\lambda) \sigma(U(\lambda)) d\lambda \\&=& C\int_{\Sn} \psi(\epsilon f^*(\omega))d\sigma.
\end{eqnarray*}
Now applying Theorem \ref{maximalstarequiv} and choosing $\epsilon$ appropriately completes the proof.
\end{proof}
We finish this section with the short proof of Theorem \ref{equivthm}. 
\begin{proof}[Proof of Theorem \ref{equivthm}]
Theorem \ref{firstequiv} gives the equivalence of (2) and (4). By definition (3) implies (1), by Fatou's Lemma (1) gives (2), and in the case that $f(x) \neq 0$ on $\Bn$, Theorem \ref{maximalstarequiv} tells us that (2) implies (3). The other case is obtained by applying the result to an appropriate translation of $f$.
\end{proof}

\section{Results with additional growth conditions on $\psi$}
A growth function $\psi$ is called $doubling$ if there exists a constant $C$ such that $\psi (2t) \leq C\psi(t)$ for all $t\in[0, \infty]$. We refer to the infimum of all such $C$ as the doubling constant of $\psi$. We will make use of the following property of doubling growth functions: if $s \geq 1$ then by choosing the smallest integer $k$ such that $\frac{s}{2^k} \leq 1$ and using the monotonicity of $\psi$ we have that $\psi(st) \leq C^k\psi(t)$ for all $t\in [0,\infty]$. 

We next prove the lemmas needed for Theorem \ref{equivdoubling}. We start by giving a family of Carleson measures on $\Bn$. 
\begin{lemma}\label{afcarleson}
Let $f$ be a $K$-quasiconformal mapping on $\Bn$ such that $f(x) \neq 0$ for all $x \in \Bn$. If $\psi$ is a growth function such that both $\psi$ and $\psi^{-1}$ are doubling then the measure $\mu$ defined by $d\mu = \frac{\psi(a_f(x)(1-|x|))}{\psi(|f(x)|)}\frac{dx}{1-|x|}$ is a Carleson measure on $\Bn$. 
\end{lemma}
\begin{proof}
The doubling properties of $\psi$ imply that there exist $p, q \geq 1$ such that
\begin{eqnarray*}
\frac{\psi(a)}{\psi(b)} \leq 2^p\left(\frac{a^p}{b^p} + \frac{a^{1/q}}{b^{1/q}}\right)
\end{eqnarray*}
for all $a,b \in (0, \infty)$. Indeed, when $b \leq a$, there exists $k \in \mathbb{N}$ such that
\begin{eqnarray*}
\psi(a)\leq C^{k+1}\psi(b), 
\end{eqnarray*}
and by choosing $p$ large enough we obtain
\begin{eqnarray*}
\frac{\psi(a)}{\psi(b)}\leq 2^p\frac{a^p}{b^p}.
\end{eqnarray*} 
The other case is obtained similarly using the doubling property of $\psi^{-1}$. Thus
\begin{gather*}
\int_{\Bn}\frac{\psi(a_f(x)(1-|x|))}{\psi(|f(x)|)}\frac{dx}{1-|x|} \leq \\\leq  2^p\left(\int_{\Bn} a_f(x)^{1/q}|f(x)|^{-1/q}(1-|x|)^{1/q -1}dx+  \int_{\Bn}  a_f(x)^p|f(x)|^{-p}(1-|x|)^{p -1}dx\right).
\end{gather*}
To show that these integrals are bounded by a constant that depends only on $n, K, p,$ and $q$, let $\epsilon>0$. By Holder's inequality
\begin{gather*}
\int_{\Bn} a_f(x)^{1/q}|f(x)|^{-1/q}(1-|x|)^{1/q-1} dx \leq \\\leq 2^p\left(\int_{\Bn} a_f(x)^n|f(x)|^{-n}(1-|x|)^{n\epsilon q}dx\right)^{1/qn}\left( \int_{\Bn}(1-|x|)^{(1/q-1-\epsilon)n/(n-1/q)}dx \right)^{(n-1/q)/n}.
\end{gather*}
We can choose $\epsilon > 0$, depending only on $q, n$, so that the latter integral converges. Since $f(x) \neq 0$ on $\Bn$, Lemma \ref{firstmoduluscons} implies that $|f(y)|^{-1} \approx |f(x)|^{-1}$ for all $y \in B_x$. Then by applying Lemma \ref{afintegralequiv}, the distortion inequality $|Df(x)|^n \leq K J_f(x)$ and a change of variables we obtain 
\begin{gather*}
\int_{\Bn} a_f(x)^n|f(x)|^{-n}(1-|x|)^{n\epsilon q}dx \leq C \int_{f(\Bn)} \frac{1}{|y|^n}(1-|f^{-1}(y)|)^{n\epsilon q} dy, 
\end{gather*}
where the constant $C$ depends on $n,K$ only. A result of Miniowitz \cite[Theorem 1]{miniowitz} shows there are constants $C,b$ depending on $n,K$ only so that 
\begin{eqnarray*}
\frac{1}{C}(1-|x|)^b \leq \frac{|f(x)|}{|f(0)|} \leq C (1-|x|)^{-b}
\end{eqnarray*}
for all $x\in \Bn$.
By integrating over $f(\Bn) \cap B(0, |f(0)|)$ and $f(\Bn) \setminus  B(0, |f(0)|)$ separately, switching to polar coordinates and inserting the Miniowitz result we have 
\begin{eqnarray*}
\int_{f(\Bn)} \frac{1}{|y|^n}(1-|f^{-1}(y)|)^{n\epsilon q} dy  \leq C \int_0^{|f(0)|} 
\frac{r^{n-1}}{r^n}\frac{r^\delta}{|f(0)|^\delta} dr + \int_{|f(0)|}^\infty 
\frac{r^{n-1}}{r^n}\frac{|f(0)|^\delta}{r^\delta} dr, 
\end{eqnarray*}
with $C$ and $0 < \delta \leq 1$ depending only on $n,K, q$. These integrals are clearly finite and give us the desired universal upper bound for the integral involving exponent $1/q$.

The estimate for the integral $\int_{\Bn}   a_f(x)^p|f(x)|^{-p}(1-|x|)^{p -1} dx$ is similar, noting that if $p > n$ then
\begin{eqnarray*}
\int_{\Bn}   \frac{a_f(x)^p(1-|x|)^{p -1} }{|f(x)|^p}dx &=& \int_{\Bn}  \frac{a_f(x)^n(1-|x|)^{p -1} }{|f(x)|^n}\frac{a_f(x)^{p-n}}{|f(x)|^{p-n}}dx \\ &\leq& C\int \frac{a_f(x)^n(1-|x|)^{n -1} }{|f(x)|^n} dx
\end{eqnarray*}
by Lemmas \ref{firstmoduluscons} and \ref{afequiv} and the assumption on $f$. 
It follows that there is a constant $M = M(n, K, q, p)$ such that 
\begin{eqnarray*}
\int_{\Bn}\frac{\psi(a_f(x)(1-|x|))}{\psi(|f(x)|)}\frac{dx}{1-|x|} \leq M,
\end{eqnarray*}
where $f$ is any map satisfying the assumptions of the Lemma.

To finish showing that $\mu$ is a Carleson measure let $g: \Bn \rightarrow \Rn$ be $K$-quasiconformal with $g(x) \neq 0$ on $\Bn$, $\omega \in \Sn$, and $r > 0$. By what we have already shown we can assume $r < 1/4$. Let $T_{r\omega}$ denote a M\"{o}bius automorphism of $\Bn$ that maps $(1-r)\omega$ to 0 and $\Bn \cap B(\omega, r)$ onto the lower hemisphere of $\Bn$. By setting $f(x) = g(T_{r\omega}(x))$ we have
\begin{gather*}
\int\limits_{\Bn\cap B(\omega,r)}\frac{\psi(a_g(y)(1-|y|))}{\psi(|g(y)|)}\frac{dy}{1-|y|} \leq \\Cr^{n-1}\int\limits_{\Bn\cap B(\omega,r)}\frac{\psi(Ca_f(T_{r\omega}(y))(1-|T_{r\omega}(y)|))}{\psi(|f(T_{r\omega}(y))|)(1-|T_{r\omega}(y)|)}J_{T_{r\omega}}dy = \\= Cr^{n-1}\int\limits_{\Bn} \frac{\psi(Ca_f(z)(1-|z|))}{\psi(|f(z)|)}\frac{dz}{1-|z|}\leq CMr^{n-1},
\end{gather*}
which is what we needed to show.
\end{proof}

\begin{lemma}\label{supaf}
Let $\psi$ be a doubling growth function and $f$ a quasiconformal mapping on $\Bn$. If 
\begin{eqnarray*}
\int_{\Bn} \psi(a_f(x)(1-|x|)) \frac{dx}{1-|x|} < \infty
\end{eqnarray*}
then $\psi(\sup_{x\in \Gamma(\omega)}(a_f(x)(1-|x|))) \in L^1(\Sn)$.
\end{lemma}
\begin{proof}
Fix $\omega \in \Sn$ and let $x\in \Gamma(\omega)$. Then there exists a constant $C$ depending on $n,K$ and the doubling constant of $\psi$ such that
\begin{eqnarray*}
\psi(a_f(x)(1-|x|)) &\leq&\frac{C}{(1-|x|)^n}\int\limits_{B_x} \psi(a_f(y)(1-|y|))dy \\&\leq& C\int\limits_{\Gamma(\omega)} \frac{\psi(a_f(y)(1-|y|))}{(1-|y|)^n}dy.
\end{eqnarray*}
Now if 
\begin{eqnarray*}
v(\omega) = \psi^{-1}\left(C\int_{\Gamma(\omega)} \psi (a_f(x) (1-|x|)) \frac{dx}{(1-|x|)^n}\right)
\end{eqnarray*}
then $\psi(v(\omega))\in L^1(\Sn)$. Indeed, given any function $u$ integrable on $\Bn$, Fubini's Theorem gives us
\begin{eqnarray*}
\int_{\Sn} \int_{\Gamma(\omega)} u(y) (1-|y|)^{1-n} dy d\sigma &=& \int_{\Bn} u(y)(1-|y|)^{1-n}\int_{\Sn}\chi_{\Gamma(\omega)}(y) d\sigma dy \\&\approx& \int_{\Bn} u(y) dy.
\end{eqnarray*}
The claim follows with $u(y) = \frac{\psi(a_f(y)(1-|y|))} {1-|y|}$. 

The estimates above showed that
\begin{eqnarray*}
\sup_{x\in \Gamma(\omega)}(a_f(x)(1-|x|))  \leq v(\omega),
\end{eqnarray*}
and so the proof is finished. 
\end{proof}

\begin{lemma}\label{doublinglemma}
Let $f$ be a quasiconformal mapping of $\Bn$ and $\psi$ a growth function that is doubling. If there is a function $v(\omega)$ such that $\psi(v(\omega)) \in L^1(\Sn)$ and
\begin{eqnarray*}
\sup_{x\in \Gamma(\omega)}d(f(x), \partial f(\Bn)) \leq Cv(\omega)
\end{eqnarray*} 
for almost every $\omega \in \Sn$ and some constant $C$, then $\psi(|f(\omega)|) \in L^1(\Sn)$.
\end{lemma}
\begin{proof} 
We can assume $f(0) = 0$. Let $U(\lambda) = \{\omega \in \Sn : f^*(\omega) > \lambda \}$. Like in the proof of Lemma \ref{carleson} we can write $U(\lambda)$ as the union of caps $S_{x_j}$
\begin{eqnarray*}
U(\lambda) = \cup S_{x_j}
\end{eqnarray*}
so that the caps have uniformly bounded overlap and 
\begin{eqnarray*}
(1-|x_j|)/C \leq d(S_{x_j}, \partial U(\lambda)) \leq C (1-|x_j|).
\end{eqnarray*}
Suppose $\omega \in S_{x_j}$ with $v(\omega) \leq \gamma$. By Lemma \ref{firstmoduluscons}, our assumption on $d(f(x), \partial f(\Bn))$, and our decomposition of $U(\lambda)$, there is a constant $C = C(n,K)$ such that $d(f(x_j), \partial f(\Bn))$, diam$f(B_{x_j}) \leq C\gamma $, and there is $\omega' \in \Sn \setminus U(\lambda)$ with $d(\omega,\omega') \leq C (1-|x_j|)$.  It follows that
\begin{eqnarray*}
|f(x_j)| \leq \lambda + C\gamma.
\end{eqnarray*}
Now let $M > 1$, $\gamma= \frac{\lambda}{(M+1)C}$ and suppose $\omega \in S_{x_j}$ with $v(\omega) \leq \gamma$ and  $|f(\omega)| > 2\lambda$. Then
\begin{eqnarray*}
|f(\omega) - f(x_j)| \geq |f(\omega)| - |f(x_j)| > \lambda - C\gamma = MC\gamma \geq Md(f(x_j), \partial\Omega), 
\end{eqnarray*}
and so by Lemma \ref{Mineq}
\begin{gather*}
\sigma(\{\omega\in S_{x_j}: |f(\omega)| > 2\lambda\; \textnormal{and}\; v(\omega) \leq \gamma \}) \\ \leq
\sigma(\{\omega\in S_{x_j}: |f(\omega) - f(x_j)| > M d(f(x_j), \partial \Omega) \}) \\
\leq C\sigma (S_{x_j}) (\log M)^{1-n}.
\end{gather*}
Note if we are in the case that $U(\lambda) = \Sn$ then $U(\lambda) = S_0$, and with our assumption that $f(0) = 0$ the inequality holds anyway. 

If $|f(\omega)| > 2\lambda$ then by continuity $\omega \in U(\lambda)$, and so
\begin{gather*}
\sigma(\{\omega \in \Sn: |f(\omega)| > 2\lambda \} \\ \leq \sigma (\{\omega \in U(\lambda): |f(\omega)| > 2\lambda \; \textnormal{and}\; v(\omega) \leq \gamma \}) + \sigma (\{\omega \in \Sn: v(\omega) > \gamma \}) \\ \leq C\sum_j \sigma (S_{x_j}) (\log M)^{1-n} + \sigma (\{\omega \in \Sn: v(\omega) > \gamma \}) \\ \leq C\sigma (U(\lambda))(\log M)^{1-n} + \sigma (\{\omega \in \Sn: v(\omega) > \gamma \}).
\end{gather*} 
Thus
\begin{eqnarray*}
\int_{\Sn} \psi(\frac{1}{2}|f(\omega)|) d\sigma = \int_0^\infty \psi'(\lambda) \sigma(\{\omega \in \Sn: |f(\omega)| > 2\lambda \})d\lambda \\ \leq \int_0^\infty \psi'(\lambda) (C\sigma (U(\lambda))(\log M)^{1-n} + \sigma (\{\omega \in \Sn: v(\omega) > \frac{\lambda}{(M+1)C} \}))d\lambda \\ = C(\log M)^{1-n}\int_{\Sn}\psi(f^*(\omega))d\sigma + \int_{\Sn} \psi((M+1)C v(\omega)) d\sigma \\ \leq C(n,K)(\log M)^{1-n}\int_{\Sn}\psi(f^*(\omega))d\sigma  + C(M, n,K, C_\psi)\int_{\Sn} \psi(v(\omega))d\sigma,
\end{eqnarray*}
where $C_\psi$ denotes the doubling constant of $\psi$.

We would like to combine the integral involving $f^*(\omega)$ with the left hand side of the inequality, but since both of the integrals could be infinite we finish the proof with a convergence argument and an application of Theorem \ref{maximalstarequiv}. If we set $f_t(x) = f(tx)$, for each $0 < t < 1$, then $\sup_{x\in \Gamma(\omega)} d(f_t(x), \partial f_t(\Bn)) \leq \sup_{x\in \Gamma(\omega)}d(f(x), \partial f(\Bn))$. Applying Theorem \ref{maximalstarequiv} to $g_t(x) = f_t(x) - y$ where $y \in \Rn \setminus f(\Bn)$ is fixed, choosing $M$ large enough and letting $t\rightarrow1$ (recall here that $f(0) = 0$) we obtain
%\begin{eqnarray*}
%\int_{\Sn} \psi(\frac{\delta }{2}|f_t(\omega)|) d\sigma &\leq& C(\log M)^{1-n}\int_{\Sn}\psi(\delta f_t^*(\omega))d\sigma  + C(M, n,K)\int_{\Sn} \psi(\delta v(\omega))d\sigma \\ &\leq& C(\log M)^{1-n}\int_{\Sn}\psi(\frac{\delta}{2} |f_t(\omega)|)d\sigma + C \int_{\Sn} \psi (\delta v(\omega)) d\sigma.
%\end{eqnarray*}
%Choosing $M$ large enough and letting $t\rightarrow1$ (recall here that $f(0) = 0$), we obtain the desired estimate
\begin{eqnarray*}
\int_{\Sn} \psi(\frac{1}{2}|f(\omega)|) d\sigma \leq C \int_{\Sn} \psi(v(\omega)) d\sigma + C < \infty.
\end{eqnarray*}
\end{proof}
These lemmas together with some of the results in the previous sections give our theorem. 
\begin{proof}[Proof of Theorem \ref{equivdoubling}]
We first show the equivalence of $(1)$ and $(3)$. Theorem \ref{equivthm} implies that (1) $\Rightarrow$ (3) for all growth functions. For the reverse implication we only need to suppose $\psi$ is doubling and that (3) holds. We claim there is a constant $C = C(n,K)$ such that 
\begin{eqnarray*}
d(f(x), \partial \Omega) \leq C f_i^*(\omega)
\end{eqnarray*}
for all $\omega\in\Sn$ and  each $x\in \Gamma(\omega)$. By Lemma \ref{firstmoduluscons} it is enough to check that this is the case for each $x = t\omega, 0 < t < 1.$ By Lemma \ref{firstmoduluscons} $f(B_x)$ contains a ball of radius $d(f(x), \partial\Omega)/C$ centered at $f(x)$, with $C$ depending only on $n, K$. So there is $y\in B_x$ such that 
\begin{eqnarray*}
d(f(x), \partial\Omega) = C|f_i(y) - f_i(x)| \leq 2Cf_i^*(\omega),
\end{eqnarray*} 
which gives the claim. Then Lemma \ref{doublinglemma} and Theorem \ref{equivthm} imply that $f\in H^\psi$.

We now show that (1) and (2) are equivalent. Assume (1) and also that $0 \notin f(\Bn)$. The measure $\mu$ given by $d\mu = \frac{\psi(a_f(x)(1-|x|))}{\psi(|f(x)|)}\frac{dx}{1-|x|}$ is a Carleson measure on $\Bn$ by Lemma \ref{afcarleson}. Then by Lemma \ref{carleson} there are absolute constants $C_1$ and $C_2$ such that
\begin{eqnarray*}
\int_{\Bn} \psi(a_f(x)(1-|x|)) \frac{dx}{1-|x|} = \int_{\Bn} \psi(|f(x)|) d\mu \leq C_1 \int_{\Sn} \psi(C_2|f(\omega)|) d\sigma.
\end{eqnarray*}
The integral on the right is finite, by Theorem \ref{equivthm}, which implies (2). 

If we assume (2) then 
\begin{align*}
\int_{\Sn}\psi(\sup_{x\in \Gamma(\omega)}(a_f(x)(1-|x|)))d\sigma < \infty
\end{align*}
by Lemma \ref{supaf}. Then $\psi(|f(\omega)|) \in L^1(\Sn)$ by Lemmas \ref{afequiv} and \ref{doublinglemma}, and finally by Theorem \ref{equivthm} we have $f\in H^\psi$.
\end{proof}
We give more details here on the examples mentioned in the introduction section regarding how Theorem~\ref{equivdoubling} may fail if either the growth function or its inverse is not doubling. In our first example let $f$ be the identity mapping on $\Bn$ and choose the growth function $\psi$ to be
\begin{eqnarray*}
\psi(t) =
\begin{cases}
\frac{1}{\log\frac{1}{t}}, & t<1/2 \\
\frac{2t}{\log2}, & t \geq 1/2.
\end{cases}
\end{eqnarray*}
This $\psi$ is doubling but $\psi^{-1}$ is not. Also, we have that $f \in H^\psi$ clearly, but
\begin{eqnarray*}
\int_{\Bn} \psi(a_f(x)(1-|x|)) \frac{dx}{1-|x|} = C + C\int_{1/2}^1\frac{1}{(1-r)\log\frac{1}{1-r}}dr,
\end{eqnarray*}
which is not finite. Thus the implication $(1) \Rightarrow (2)$ fails for this example. 

Now let $f(z) = \log (z+1), z\in\mathbb{B}^2$, and $\psi (t) = e^{t^2} -1, t\in [0,\infty]$. Then $\psi$ is a growth function that is not doubling. The second coordinate function of $f$ is a bounded function, and so $\psi(f^*_2(\omega)) \in L^1(\mathbb{S}^1)$. However, since
\begin{eqnarray*}
M(r,f) \geq \log \frac{1}{1-r}
\end{eqnarray*}
for all $r \in (0,1)$ and
\begin{eqnarray*}
\int_0^1 \psi\left(\delta \log \frac{1}{1-r}\right) dr
\end{eqnarray*}
diverges given any $\delta$, $f\notin H^\psi$ by Theorem \ref{equivthm}. This example shows how the implication $(3) \Rightarrow (1)$ can fail when the growth function is not doubling. 

Finally, let $f(z) = \log (z+1), z\in\mathbb{B}^2$, and choose

\begin{eqnarray*}
\psi(t) =
\begin{cases}
2et, & t\leq1 \\
e^{t^2}+e, & t > 1,
\end{cases}
\end{eqnarray*}
as the growth function. This $\psi$ is not doubling, and, like above, $f\notin H^\psi$ by Theorem \ref{equivthm}. However,
\begin{eqnarray*}
\int_{\Bn} \psi(a_f(x)(1-|x|))\frac{dx}{1-|x|} &=& \int_{\Bn}\psi\left(\frac{1-|x|}{|x+1|}\right)\frac{dx}{1-|x|} \\&=& \int_{\Bn}\frac{dx}{|x+1|} = \int_{B(1,1)}\frac{dx}{|x|} < \infty,
\end{eqnarray*}
 which shows that (2) does not imply (1) for this example.
 \vspace{3mm}
 
 \noindent\textit{Acknowledgement.} This paper forms a part of the thesis of the author, written under the supervision of Pekka Koskela. 
 
\bibliographystyle{plain}
\bibliography{secondpaperref}

 \noindent  Sita Benedict\\
Department of Mathematics and Statistics\\
University of Jyv\"{a}skyl\"{a}\\
P.O. Box 35 (MaD)\\
FI-40014\\
Finland\\

\noindent{\it E-mail address}: \texttt{sita.c.benedict@jyu.fi}

\end{document}